\documentclass[a4paper,11pt]{article}
\usepackage{hyperref}
\usepackage{geometry}
\usepackage{amssymb}
\usepackage{amsmath}
\usepackage{amsthm}
\usepackage{authblk}
\usepackage{paralist}
\usepackage{titlefoot}
\usepackage[numbers]{natbib} % has a nice set of citation styles and commands

\newtheorem{theorem}{Theorem}[section]
\newtheorem*{theorem*}{Theorem}
\newtheorem{remark}[theorem]{Remark}
\newtheorem{definition}[theorem]{Definition}
\newtheorem*{definition*}{Definition}

\newtheorem{lemma}[theorem]{Lemma}
\newtheorem*{lemma*}{Lemma}
\newtheorem{corollary}[theorem]{Corollary}
\newtheorem*{corollary*}{Corollary}
\newtheoremstyle{boldremark}
    {\dimexpr\topsep/2\relax} % space above
    {\dimexpr\topsep/2\relax} % space below
    {}          % body font
    {}          % indent amount
    {\bfseries} % theorem head font
    {.}         % punctuation after theorem head
    {.5em}      % space after theorem head
    {}          % theorem hed spec. (empty = "normal")
\theoremstyle{boldremark}

\newtheorem{assumption}[theorem]{Assumption}
\newtheorem*{assumption*}{Assumption}

\newcommand{\R}{\mathbb{R}}

\renewcommand{\P}{\mathbb{P}}
\newcommand{\E}{\mathbb{E}}
\newcommand{\A}{\mathbb{A}}

\newcommand{\cF}{\mathcal{F}}

\newcommand{\cP}{\mathcal{P}}

\newcommand{\cL}{\mathcal{L}}

\renewcommand{\epsilon}{\varepsilon}

\newcommand{\dt}{\,\mathrm{d}t}

\newcommand{\dlambda}{\,\mathrm{d}\lambda}
\newcommand{\dgamma}{\,\mathrm{d}\gamma}
\newcommand{\dW}{\,\mathrm{d}W}

\begin{document}

% \Header

\title{A Novel Approach to Peng's Maximum Principle for McKean-Vlasov Stochastic Differential Equations}
\date{\today}
	
\author[1]{Johan Benedikt Spille}
\author[1]{Wilhelm Stannat}
	
\affil[1]{Technische Universität Berlin, Berlin, Germany}
	
% \affil{\small Technische Universit\"at Berlin, Berlin, Germany}
	
\maketitle

\unmarkedfntext{\textit{Mathematics Subject Classification (2020) --- 
93E20, %Optimal stochastic control
49K45, %Optimality conditions for problems involving randomness
49N80, %Mean field games and control
60H30, %Applications of stochastic analysis
60H10, %Stochastic ordinary differential equations
49N15 % Duality theory (optimization)
}}
	
\unmarkedfntext{\textit{Keywords and phrases --- 
Stochastic control,
Maximum principle,
McKean-Vlasov equation,
Mean-field SDE,
Duality in optimization,
Adjoint calculus}}
	
\unmarkedfntext{\textit{Mail}: \textbullet \href{mailto:spille@math.tu-berlin.de}{spille@math.tu-berlin.de},\,\textbullet \href{mailto:stannat@math.tu-berlin.de}{stannat@math.tu-berlin.de}}

\begin{abstract}
    We present a novel approach to the proof of Peng's maximum principle for McKean-Vlasov stochastic differential equations (SDE). The main step is the introduction of a third adjoint equation, a conditional McKean-Vlasov backward SDE, to accommodate the dualization of quadratic terms containing two independent copies of the first-order variational process. This is an intrinsic extension of the maximum principle from Peng for standard SDE and gives a conceptually consistent proof. 
    Our approach will be useful in further extensions to the common noise setting and the infinite dimensional setting.
\end{abstract}

%\newpage
\tableofcontents

\section{Introduction}
In this paper, we are concerned with the following optimization problem: Minimize over all admissible controls the cost functional
\begin{equation}\label{eq:cost}
    J(\alpha)
    = \mathbb{E}\left[
        \int_0^T f(t,X_t,\mu_t,\alpha_t)\dt
        + g(X_T,\mu_T)
    \right],
\end{equation}
where we fix a terminal time $T>0$ and $f:[0,T]\times \R^d\times \cP_2(\R^d)\times U\to \R$ and $g:\R^d\times \cP_2(\R^d)\to \R$ are deterministic functions, subject to the state equation, given by the controlled McKean-Vlasov SDE
\begin{equation}\label{eq:state}
    dX_t = A(t,X_t,\mu_t,\alpha_t)\dt 
          + B(t,X_t,\mu_t,\alpha_t)\dW_t,
    \qquad X_0 = x_0,
\end{equation}
where $\mu_t := \mathcal{L}(X_t)$ denotes the distribution of $X_t$, 
$(A,B):[0,T]\times \R^d\times \cP_2(\R^d)\times U\to \R^d\times\R^{d\times d}$ are deterministic and 
$W$ is a $d$-dimensional Brownian motion on a complete filtered probability space $(\Omega,\mathcal{F},\mathbb{F},\mathbb{P})$ such that $\mathbb{F}=(\mathcal{F}_t)_{0\le t\le T}$ is the augmented filtration generated by $W$. Here, an admissible control process $\alpha=(\alpha_t)_{0\le t\le T}$ takes values in a non-empty, not necessarily convex set $U\subset\mathbb{R}^m$, is progressively measurable and $\mathbb{E}[\int_0^T |\alpha_t|^2\dt] < \infty.$ The set of all admissible controls is denoted by $\mathbb{A}$.

The general goal is to deduce a necessary Pontryagin-type optimality condition, first developed by \citet{Pontryagin1962} in the deterministic setting. Later, this method was extended by \citet{Bismut1978} to SDEs. 
The generalization to non-convex control domains requires in the stochastic setting a second-order Taylor expansion of the cost functional. To this end, \citet{Peng1990} introduced a second-order adjoint state and derived a corresponding maximum principle, nowadays known as Peng's maximum principle.

In many applications the dynamics of the controlled process depend not only on its current state but also on its distribution. These dynamics are modeled by McKean-Vlasov SDEs like \eqref{eq:state}, originally studied in the context of interacting particle systems \cite{Kac1956,McKean1967,Sznitman1991} and now fundamental in mean-field games and large population control problems \cite{LasryLions2007,CarmonaDelarue2018I}. The optimal control of McKean-Vlasov dynamics has therefore attracted significant attention in recent years. 

Early results use convex control domains, and derive Pontryagin maximum principles adapted to the McKean-Vlasov setting \cite{Andersson2011,CarmonaDelarue2015,CarmonaDelarue2018I,HocquetVogler2020}. Such maximum principles have also been generalized to infinite dimensional settings \cite{Ahmed2016,Dumitrescu2018,Tang2019,spille2025}. In comparison Peng's maximum principle for McKean-Vlasov SDEs has only been treated in finite dimensions so far. 
These extensions to non-convex control domains with mean-field dependence were done by \citet{Buckdahn2011,Buckdahn2016,BuckdahnPeng2017}. 

The main goal of this paper now, is not to provide a new maximum principle. Instead we provide a novel approach to the proof of the maximum principle given by \citet{Buckdahn2016}. 
Following the lines in \citet{Peng1990}, the second-order Taylor expansion of the cost functional involves quadratic terms coming from the (mixed) Lions derivative with respect to $\mu_t$, which cannot be dualized in the same way as the quadratic terms coming from the derivative with respect to the state (cf. Remark \ref{remark:Discussion at the end}). \citet{Buckdahn2016} observed that these terms are in fact of lower-order, since they involve taking suitable expectations, and, thus, do not contribute to the maximum principle. 
Clearly, this method does not work anymore as soon as these estimates do not hold (cf. applications below). Thus, we aim to give a more general and also conceptually consistent proof.

To this end, we introduce a third adjoint state \eqref{eq:Second-Level Second-Order Adjoint} that, in particular, dualizes quadratic terms coming from the second order (mixed) Lions derivatives. 
In contrast to the second adjoint equation, which is standard BSDE, the third adjoint equation is a conditional McKean-Vlasov BSDE.
While the second adjoint state can be seen as dualizing the homogeneous part of the first variational equation, the third adjoint equation dualizes the heterogeneous part. Importantly, the third adjoint process, does not appear in the maximum principle, it is only used as a means to get there. It will, however, appear as soon as the sharper estimates from \cite{Buckdahn2016} Proposition 4.3 do no-longer hold.

Our approach will be useful in extensions to the maximum principle to an infinite dimensional setting and the common noise setting.
Firstly, the infinite dimensional setting. The proof of the sharper estimates in \citet{Buckdahn2016} relies on matrix valued SDEs and needs apriori estimates on solutions to these. In infinite dimensions these become operator valued SDEs, which cannot be solved in a general setting. Some progress has been made by \citet{guatteri2025pengsmaximumprinciplestochastic}, but in a simpler setting and relying on solving BSPDEs as evolution equations on the space of Hilbert-Schmidt operators. A difficulty of our approach in the infinite dimensional setting would be the existence of a solution to the third adjoint equation as it becomes an operator valued backward SDE. But this is the same problem that is already present for the second adjoint and ways to solve these equations have been developed \cite{WesselsPengMaximumPrinciple2021,Lu2021}.

The other clear application is the common noise setting \cite{CarmonaDelarue2018II}. Here, it is already known that these sharper estimates of \citet{Buckdahn2016} Proposition 4.3 do not hold, so the corresponding second-order terms have to appear, e.g. compare the different Ito formulas \cite{CarmonaDelarue2018I} Proposition 5.102 and \cite{CarmonaDelarue2018II} Theorem 4.17. The intuition is that in general a conditional expectation does not have the same smoothing properties as a non-conditional expectation. Our method can still be applied in this setting.

The structure of this paper is as follows. In Section 2, we introduce the mathematical setting and the needed assumptions. In Section 3, we introduce the variational processes and give estimates on their order. In Section 4, we introduce the adjoint equations, in particular the third adjoint equation and its needed setting and prove its well-posedness. Afterwards, in Section 5, the needed dualizations of these processes are discussed and then used in Section 6 to expand the cost functional in a suitable way to infer the desired maximum principle.

\section{Preliminaries}
We recall the notion of Lions differentiability. Let $\mathcal{P}_2(\mathbb{R}^d)$ denote the $2$-Wasserstein space (for more detail see \cite{Villani2009}).  
A mapping $\varphi:\mathcal{P}_2(\mathbb{R}^d)\to\mathbb{R}^m$ is said to be Lions differentiable at $\mu$ if there exists a mapping 
$\partial_\mu\varphi(\mu):\mathbb{R}^d\to\mathbb{R}^{m\times d}$ such that, for any $X\sim\mu$, the lifted mapping $\tilde{\varphi}: L^2(\Omega,\cF,\P;\R^d)\mapsto \R^{m\times d},X\mapsto \varphi(\mathcal{L}(X))$ is Fréchet differentiable in $L^2(\Omega,\cF,\P;\R^d)$ and
\begin{equation}
    D\tilde\varphi(X) Y 
    = \mathbb{E}\left[\partial_\mu\varphi(\mu)(X) Y
    \right],
    \qquad \text{for all }Y\in L^2(\Omega,\cF,\P;\R^d).
\end{equation}
For more details we refer to \cite{CarmonaDelarue2018I}.
We now introduce the necessary regularity for the maximum principle as done in \cite{BuckdahnPeng2017} Definition 2.1 (also compare \cite{Buckdahn2016} Definition 2.1 and \cite{CarmonaDelarue2018I} Chapter 5.6).
\begin{definition}
    We say that $\varphi \in C_b^{1,1}(\mathcal{P}_2(\mathbb{R}^d),\R^m)$,
    if there exists for all $\vartheta \in L^2(\Omega, \mathbb{R}^d)$ a $\cL(\vartheta)$-modification of $\partial_\mu \varphi(\cL(\vartheta))$, again denoted by $\partial_\mu \varphi(\cL(\vartheta))$, such that $\partial_\mu \varphi: \mathcal{P}_2(\mathbb{R}^d) \times \mathbb{R}^d \to \mathbb{R}^{m\times d}$ is bounded and Lipschitz continuous, i.e. there is some $C>0$ such that
    \begin{enumerate}[(i)]
        \item $|\partial_\mu \varphi(\mu)(y)| \leq C$, for all $ \mu \in \mathcal{P}_2(\mathbb{R}^d)$ and $ y \in \mathbb{R}^d$, and
        \item $|\partial_\mu \varphi(\mu)(y)-\partial_\mu \varphi(\mu^{\prime})(y^{\prime})| \leq C(W_2(\mu, \mu^{\prime})+|y-y^{\prime}|)$, for all $ \mu, \mu^{\prime} \in \mathcal{P}_2(\mathbb{R}^d)$ and $ y,y^{\prime} \in \mathbb{R}^d$.
    \end{enumerate}
\end{definition}
This version of the derivative is unique (cf. \cite{BuckdahnPeng2017} Remark 2.1) and we will always be using this version. Note that for $(\mu,y)\mapsto \partial_\mu\varphi(\mu)(y)$ we now have two different partial derivatives $\partial_{y}\partial_\mu \varphi$ and $\partial_{\mu}\partial_\mu \varphi$. We give further regularity definitions and a second-order Taylor expansion in the measure argument, which will be needed later. 
\begin{definition}
    We say that $\varphi \in C_b^{2,1}(\mathcal{P}_2(\mathbb{R}^d),\R^m)$, if $\varphi \in C_b^{1,1}(\mathcal{P}_2(\mathbb{R}^d),\R^m)$,
    \begin{enumerate}[(i)]
        \item $\partial_\mu \varphi(\cdot)( y) \in C_b^{1,1}(\mathcal{P}_2(\mathbb{R}^d),\R^{m\times d})$, for all $y \in \mathbb{R}^d$, and $\partial_\mu^2 \varphi$ : $\mathcal{P}_2(\mathbb{R}^d) \times \mathbb{R}^d \times \mathbb{R}^d \to \mathbb{R}^{m\times d\times d}$ is bounded and Lipschitz-continuous and
        \item $\partial_\mu \varphi(\mu): \mathbb{R}^d \to \mathbb{R}^{m\times d}$ is differentiable, for every $\mu \in \mathcal{P}_2(\mathbb{R}^d)$, and its derivative $\partial_y \partial_\mu \varphi: \mathcal{P}_2(\mathbb{R}^d) \times \mathbb{R}^d \to \mathbb{R}^{d\times d\times d}$ is bounded and Lipschitz-continuous.
    \end{enumerate}
\end{definition}
\begin{lemma}[\cite{BuckdahnPeng2017} Lemma 2.1] \label{lemma:Taylor formula for Lions Derivative}
    Let $\varphi \in C_b^{2,1}(\mathcal{P}_2(\mathbb{R}^d),\R)$. Then, for any given $\vartheta_0 \in L^2(\Omega, \mathbb{R}^d)$ we have the following second-order expansion, for all $\vartheta\in L^2(\Omega, \mathbb{R}^d)$
    \begin{equation*}
        \begin{aligned}
            \varphi(\cL(\vartheta))-  \varphi(\cL(\vartheta_0))
            = & \E[\partial_\mu \varphi(\cL(\vartheta_0))( \vartheta_0) \cdot \eta]+\frac{1}{2} \E[\tilde{\E}[\operatorname{tr}(\partial_\mu^2 \varphi(\cL(\vartheta_0),)(\tilde{\vartheta_0}, \vartheta_0) \cdot \tilde{\eta} \otimes \eta)]] \\
            & +\frac{1}{2} \E[\operatorname{tr}(\partial_y \partial_\mu \varphi(\cL(\vartheta_0))( \vartheta_0) \cdot \eta \otimes \eta)]+R(\cL(\vartheta), \cL(\vartheta_0)),
        \end{aligned}
    \end{equation*}
    where $\eta:=\vartheta-\vartheta_0$, and for all $\vartheta \in L^2(\Omega, \mathbb{R}^d)$ the remainder $R(\cL(\vartheta), \cL(\vartheta_0))$ satisfies the estimate
    \begin{equation*}
        \begin{aligned}
            |R(\cL(\vartheta), \cL(\vartheta_0))| & \leq C((\E[|\vartheta-\vartheta_0|^2])^{3 / 2}+\E[|\vartheta-\vartheta_0|^3])
            \leq C \E[|\vartheta-\vartheta_0|^3] .
        \end{aligned}
    \end{equation*}
    The constant $C \in \mathbb{R}_{+}$only depends on the Lipschitz constants of $\partial_\mu^2 f$ and $\partial_y \partial_\mu \varphi$.
\end{lemma}
\begin{remark}
    Note that the above Lemma gives a second-order Fréchet derivative of the lift to the space $L^3(\Omega,\R^d)$ (assuming one has sufficient integrability of the corresponding laws), as the remainder has sufficient order in $L^3$. This relates to the classical problem that the assumption that the lift $\tilde{\varphi}:L^2(\Omega,\R^d)\to \R$ is two times Fréchet differentiable (we do not assume this!) is very strong (cf. \cite{CarmonaDelarue2018I} Example 5.80).
\end{remark}
For our coefficients we will need the following regularity.
\begin{definition}
    For $\varphi:\R^d\times\cP_2(\R^d)\to \R^m $ we say $\varphi\in C_b^{2,1}(\R^d\times\cP_2(\R^d),\R^m)$ if
    \begin{enumerate}[(i)]
        \item $\varphi(\cdot ,\mu)\in C^2_b(\R^d,\R^m)$ for all $\mu\in\cP_2(\R^d)$,
        \item $\varphi(x,\cdot)\in C^{2,1}_b(\cP_2(\R^d),\R^m)$ and $\partial_x\varphi(x,\cdot)\in C^{1,1}_b(\cP_2(\R^d),\R^{m\times d})$ for all $x\in\R^d$,
        \item $\partial_\mu\varphi(\cdot ,\mu)(\cdot)\in C^1_b(\R^d\times \R^d,\R^{m\times d})$ for all $\mu\in\cP_2(\R^d)$ and
        \item $\varphi$ and all first and second-order derivatives of $\varphi$ are bounded and Lipschitz continuous in all variables.
    \end{enumerate}
\end{definition}
We will use shorthand notations for our derivatives. For $\varphi \in C_b^{2,1}(\R^d\times\cP_2(\R^d),\R^m)$ define $\varphi_x:=\partial_x\varphi$, $\varphi_\mu:=\partial_\mu \varphi$, $\varphi_{xx}:=\partial_x\partial_x\varphi$. Further, for $(x,\mu,y)\mapsto \partial_\mu\varphi(x,\mu)(y)$, we denote $\varphi_{x\mu}:=\partial_{x}\partial_\mu \varphi$, $\varphi_{y\mu}:=\partial_{y}\partial_\mu \varphi$ and $\varphi_{\mu\mu}:=\partial_{\mu}\partial_\mu \varphi$.
\begin{assumption}\label{Assumption:C^{2,1}}
    The coefficients $A,B,f,g$ are measurable in all variables and for all $t\in[0,T]$ and $u\in U$ 
    \begin{enumerate}[(i)]
        \item $A(t,\cdot,\cdot,u)\in C_b^{2,1}(\R^d\times\cP_2(\R^d),\R^d)$,
        \item $B(t,\cdot,\cdot,u)\in C_b^{2,1}(\R^d\times\cP_2(\R^d),\R^{d\times d})$,
        \item $f(t,\cdot,\cdot,u)\in C_b^{2,1}(\R^d\times\cP_2(\R^d),\R)$ and 
        \item $g\in C_b^{2,1}(\R^d\times\cP_2(\R^d),\R)$,
    \end{enumerate}
    with Lipschitz and boundedness constants uniform in $t$ and $u$.
\end{assumption}
Note that under these assumptions the solution to the state equation \eqref{eq:state} exists and is unique for every $\alpha\in \A$ by standard results \cite{CarmonaDelarue2018I,Sznitman1991}.

For brevity's sake, we will often write $\theta_t:=(t,X_t,\mu_t,\alpha_t)$, where $\alpha$ will be optimal, $X$ the corresponding solution to \eqref{eq:state} and $\mu$ its law. 

Moreover, we will often denote the transpose of a matrix $A\in\R^{d\times d}$ as its dual operator $A^*=A^\top$ and use the tensor product $x\otimes y:=xy^\top$ for $x,y\in \R^d$.

Further, we will be interested in the order of processes with respect to some parameter. For this we introduce the $\mathcal{O}$-notation: For mappings $f:(0,\infty)\to \R^m$, $g:(0,\infty)\to \R$, we write $f\in O(g(\epsilon))$ if $\limsup_{\epsilon\to 0} \left|\frac{f(\epsilon)}{g(\epsilon)}\right|< \infty$ and $f\in o(g(\epsilon))$ if $\lim_{\epsilon\to 0} \left|\frac{f(\epsilon)}{g(\epsilon)}\right|=0$.

\section{The Variational Equations}

Let $\alpha\in \A$ be an optimal control with corresponding state process $X$. For a Borel subset $E_\epsilon\subset [0,T]$ with Lebesgue measure $|E_\epsilon|=\epsilon>0$ and another admissible control $\beta\in \A$ define the spike variation
\begin{equation*}
    \alpha_t^{\epsilon} :=\begin{cases}
    \beta_t,&\text{for } t \in E_{\epsilon}, \\
    \alpha_t,&\text{for }  t \in E_{\epsilon}^c.
    \end{cases}
\end{equation*}
We denote $X^\epsilon$ the solution to the state equation \eqref{eq:state} with $\alpha^\epsilon$ plugged in, $\mu^\epsilon$ the corresponding distribution and $\Delta X = X^\epsilon-X$. Further, we denote $\delta A(t) = A(t,X_t,\mu_t,\alpha^\epsilon_t)-A(t,X_t,\mu_t,\alpha_t)$ and similarly for any other mapping. Then, we define the first variational process $Y^\epsilon$ as a solution to
\begin{equation}
    \begin{aligned}
        dY^\epsilon_t 
        =& \left( A_x(\theta_t)Y^\epsilon_t +\tilde{\E}\left[A_\mu(\theta_t)(\tilde{X}_t) \tilde{Y}^\epsilon_t\right] 
        + \delta A(t) \right) \dt  \\
        &+ \left( B_x(\theta_t)Y^\epsilon_t + \tilde{\E}\left[B_\mu(\theta_t)(\tilde{X}_t) \tilde{Y}^\epsilon_t\right] 
        + \delta B(t) \right) \dW_t  \\
        Y^\epsilon_0 =& 0,
    \end{aligned}\label{eq:first variational process SDE}
\end{equation}
and the second variational process $Z^\epsilon$ as a solution to
\begin{equation}
    \begin{aligned}
        dZ^\epsilon_t 
        = & \Big( A_x(\theta_t)Z^\epsilon_t +\tilde{\E}\left[A_\mu(\theta_t)(\tilde{X}_t) \tilde{Z}^\epsilon_t\right] + \frac{1}{2} A_{xx}(\theta_t)[Y^\epsilon_t,Y^\epsilon_t] \\
        &+ \frac{1}{2}\tilde{\E}\left[ A_{y\mu}(\theta_t)(\tilde{X}_t)[\tilde{Y}^\epsilon_t,\tilde{Y}^\epsilon_t]\right]
        + \frac{1}{2}\tilde{\E} \left[\tilde{\tilde\E}\left[ A_{\mu\mu}(\theta_t)(\tilde{X}_t,\tilde{\tilde{X}}_t)[\tilde{Y}^\epsilon_t,\tilde{\tilde{Y}}^\epsilon_t]\right]\right] \\
        &+\tilde{\E}\left[ A_{x\mu}(\theta_t)(\tilde{X}_t)[\tilde{Y}^\epsilon_t,Y^\epsilon_t]\right]
        + \delta A_x(t) Y^\epsilon_t+\tilde{\E}\left[\delta A_\mu(t)(\tilde{X}_t) \tilde{Y}^\epsilon_t\right]  \Big)\dt  \\
        &+\Big( B_x(\theta_t)Z^\epsilon_t +\tilde{\E}\left[B_\mu(\theta_t)(\tilde{X}_t) \tilde{Z}^\epsilon_t\right] 
        + \frac{1}{2} B_{xx}(\theta_t)[Y^\epsilon_t,Y^\epsilon_t] \\
        &+ \frac{1}{2}\tilde{\E}\left[ B_{y\mu}(\theta_t)(\tilde{X}_t)[\tilde{Y}^\epsilon_t,\tilde{Y}^\epsilon_t]\right]
        +\frac{1}{2}\tilde{\E} \left[\tilde{\tilde\E}\left[ B_{\mu\mu}(\theta_t)(\tilde{X}_t,\tilde{\tilde{X}}_t)[\tilde{Y}^\epsilon_t,\tilde{\tilde{Y}}^\epsilon_t]\right]\right] \\
        &+\tilde{\E}\left[ B_{x\mu}(t)(\tilde{X}_t)[\tilde{Y}^\epsilon_t,Y^\epsilon_t]\right]
        + \delta B_x(t) Y^\epsilon_t+\tilde{\E}\left[\delta B_\mu(t)(\tilde{X}_t) \tilde{Y}^\epsilon_t\right]\Big)  \dW_t\\
        Z^\epsilon_0 =& 0. 
    \end{aligned}\label{eq:second variational process SDE}
\end{equation}
\begin{remark}
    Note that these solutions exist and are unique by Assumption \ref{Assumption:C^{2,1}} and standard results (cf. \cite{CarmonaDelarue2018I,Sznitman1991}). Even though \eqref{eq:first variational process SDE} and \eqref{eq:second variational process SDE} are not true McKean-Vlasov SDEs if looked at in a  stand-alone sense as they depend on the joint distribution of $(X,Y)$ and $(X,Z)$ which is stronger than the distribution of $Y$ and $Z$ respectively, one can still use standard theory by looking at the corresponding equation of $(X,Y,Z)$ on $\R^{3d}$ as a whole.
\end{remark}
We get the following estimates on the order of these processes with respect to $\epsilon$.
\begin{lemma}\label{lemma:Estimates on Y and Z}
    If Assumption \ref{Assumption:C^{2,1}} holds, then, for any $k \geq 1$,
    \begin{align}
        &\mathbb{E}\left[\sup _{t \in[0, T]}\left\|\Delta X_t\right\|^{2 k}\right] \in O(\epsilon^k), \label{eq:DeltaXEstimate}\\
        &\mathbb{E}\left[\sup _{t \in[0, T]}\left\|Y^\epsilon_t\right\|^{2 k}\right] \in O(\epsilon^k), \label{eq:FirstVariationalEstimate}\\
        &\mathbb{E}\left[\sup _{t \in[0, T]}\left\|Z^\epsilon_t\right\|^{2 k}\right] \in O(\epsilon^{2k}), \label{eq:SecondVariationalEstimate} \\
        &\mathbb{E}\left[\sup _{t \in[0, T]}\left\|\Delta X_t-Y^\epsilon_t\right\|^{2 k}\right] \in O(\epsilon^{2k}). \label{eq:DeltaX minus first Variational Estimate}
    \end{align}
\end{lemma}
\begin{proof}
     For equations \eqref{eq:DeltaXEstimate}, \eqref{eq:FirstVariationalEstimate}, \eqref{eq:DeltaX minus first Variational Estimate}, the proof is the same as for Proposition 4.2 in \cite{Buckdahn2016} as no new terms appear in the equations of $X$, $X^\epsilon$ and $Y^\epsilon$. The proof of \eqref{eq:SecondVariationalEstimate} is done similarly by a simple Gronwall argument and using \eqref{eq:FirstVariationalEstimate}. 
\end{proof}
\begin{lemma}\label{lemma:estimate on DeltaX-Y-Z}
    If Assumption \ref{Assumption:C^{2,1}} holds, then, for any $ k\geq 1$,
    \begin{equation*}
        \mathbb{E}\left[\sup _{t \in[0, T]}\left|X^{\epsilon}_t-X_t-Y^{\epsilon}_t-Z^{\epsilon}_t\right|^{2 k}\right] \in o(\epsilon^{2k}).
    \end{equation*}
\end{lemma}
\begin{proof}
    For notational convenience denote 
    $K^\epsilon_t:=X^{\epsilon}_t-X_t-Y^{\epsilon}_t-Z^{\epsilon}_t$. Clearly, by our previous estimates from Lemma \ref{lemma:Estimates on Y and Z}
    \begin{equation}
        \E\left[\sup_{t\in[0,T]}\|K^\epsilon_t\|^{2k}\right]\in O(\epsilon^{2k}).\label{eq:Apriori Estimate on K}
    \end{equation}
    Defining
    \begin{equation}\begin{aligned}
        K^{(1)}_t
        =& A(t,X^\epsilon_t,\mu^\epsilon_t,\alpha^\epsilon_t)
        -A(t,X_t,\mu_t,\alpha^\epsilon_t)
        - A_x(\theta_t)[Y^\epsilon_t+Z^\epsilon_t] \\
        &-\tilde{\E}\left[A_\mu(\theta_t)(\tilde{X}_t) [\tilde{Y}^\epsilon_t+\tilde{Z}^\epsilon_t]\right] 
        - \frac{1}{2} A_{xx}(\theta_t)[Y^\epsilon_t,Y^\epsilon_t] \\
        &- \frac{1}{2}\tilde{\E}\left[ A_{y\mu}(\theta_t)(\tilde{X}_t)[\tilde{Y}^\epsilon_t,\tilde{Y}^\epsilon_t]\right] 
        - \frac{1}{2}\tilde{\E} \left[\tilde{\tilde\E}\left[ A_{\mu\mu}(\theta_t)(\tilde{X}_t,\tilde{\tilde{X}}_t)[\tilde{Y}^\epsilon_t,\tilde{\tilde{Y}}^\epsilon_t]\right]\right] \\
        &-\tilde{\E}\left[ A_{x\mu}(\theta_t)(\tilde{X}_t)[\tilde{Y}^\epsilon_t,Y^\epsilon_t]\right]
        - \delta A_x(t) Y^\epsilon_t
        -\tilde{\E}\left[\delta A_\mu(t)(\tilde{X}_t) \tilde{Y}^\epsilon_t\right]  
    \end{aligned}\label{eq:Drift of K}\end{equation}
    and $K^{(2)}$ the corresponding terms for $A$ replaced by $B$, we can write
    \begin{equation*}
        dK^\epsilon_t= K^{(1)}_t \dt + K^{(2)}_t \dW_t.
    \end{equation*}
    Using the notation $X^{\lambda, \epsilon}:=X+\lambda\left(X^\epsilon-X\right)$ and $\theta^{\lambda,\epsilon}=(t,X^{\lambda,\epsilon},\cL(X^{\lambda,\epsilon}),\alpha^\epsilon)$, we notice
    \begin{equation}\begin{aligned}
        &A(t,X^\epsilon_t,\mu^\epsilon_t,\alpha^\epsilon_t)
        -A(t,X_t,\mu_t,\alpha^\epsilon_t)
        - A_x(\theta_t)[Y^\epsilon_t+Z^\epsilon_t] 
        -\tilde{\E}[A_\mu(\theta_t)(\tilde{X}_t) [\tilde{Y}^\epsilon_t+\tilde{Z}^\epsilon_t]]\\
        =&\int_0^1 A_x(\theta^{\lambda,\epsilon}_t)[K^\epsilon_t]
        +\tilde{\E}[A_\mu(\theta^{\lambda,\epsilon}_t)(\tilde{X}^{\lambda,\epsilon}_t)[\tilde{K}^\epsilon_t]]
        +(A_x(\theta^{\lambda,\epsilon}_t)-A_x(\theta_t))[Y^\epsilon_t+Z^\epsilon_t]\\
        &+\tilde{\E}[(A_\mu(\theta^{\lambda,\epsilon}_t)(\tilde{X}^{\lambda,\epsilon}_t)-A_\mu(\theta_t)(\tilde{X}_t))[\tilde{Y}^\epsilon_t+\tilde{Z}^\epsilon_t]]\dlambda.
    \end{aligned}\label{eq:first expansion of drift of K}\end{equation}
    Expanding further, we get
    \begin{equation}\begin{aligned}
        &A_x(\theta^{\lambda,\epsilon}_t)-A_x(\theta_t)\\
        =& \lambda \int_0^1 A_{xx}(\theta^{\lambda\gamma,\epsilon}_t)[K^\epsilon_t]
        +\tilde{\E}\left[A_{\mu x}(\theta^{\lambda\gamma,\epsilon}_t)(\tilde{X}^{\lambda\gamma,\epsilon}_t)[\tilde{K}^\epsilon_t]\right]\dgamma +\delta A_x(t)\\
        &+\lambda \int_0^1 A_{xx}(\theta^{\lambda\gamma,\epsilon}_t)[Y^\epsilon_t+Z^\epsilon_t]
        +\tilde{\E}\left[A_{\mu x}(\theta^{\lambda\gamma,\epsilon}_t)(\tilde{X}^{\lambda\gamma,\epsilon}_t)[\tilde{Y}^\epsilon_t+\tilde{Z}^\epsilon_t]\right]\dgamma,
    \end{aligned}\label{eq:second expansion of drift of K in x}\end{equation}
    and
    \begin{equation}\begin{aligned}
        &A_\mu(\theta^{\lambda,\epsilon}_t)(\tilde{X}^{\lambda,\epsilon}_t)-A_\mu(\theta_t)(\tilde{X}_t)\\
        =& \lambda \int_0^1 \tilde{\tilde{\E}}\left[A_{\mu\mu}(\theta^{\lambda\gamma,\epsilon}_t)(\tilde{X}^{\lambda\gamma,\epsilon}_t,\tilde{\tilde{X}}^{\lambda\gamma,\epsilon}_t)[\tilde{\tilde{K}}^\epsilon_t]\right]
        +A_{x\mu}(\theta^{\lambda\gamma,\epsilon}_t)(\tilde{X}^{\lambda\gamma,\epsilon}_t)[K^\epsilon_t]\\
        &+A_{y\mu}(\theta^{\lambda\gamma,\epsilon}_t)(\tilde{X}^{\lambda\gamma,\epsilon}_t)[\tilde{K}^\epsilon_t]
        +\tilde{\tilde{\E}}\left[A_{\mu\mu}(\theta^{\lambda\gamma,\epsilon}_t)(\tilde{X}^{\lambda\gamma,\epsilon}_t,\tilde{\tilde{X}}^{\lambda\gamma,\epsilon}_t)[\tilde{\tilde{Y}}^\epsilon_t+\tilde{\tilde{Z}}^\epsilon_t]\right]\\
        &+A_{x\mu}(\theta^{\lambda\gamma,\epsilon}_t)(\tilde{X}^{\lambda\gamma,\epsilon}_t)[Y^\epsilon_t+Z^\epsilon_t]
        +A_{y\mu}(\theta^{\lambda\gamma,\epsilon}_t)(\tilde{X}^{\lambda\gamma,\epsilon}_t)[\tilde{Y}^\epsilon_t+\tilde{Z}^\epsilon_t]\dgamma 
        +\delta A_\mu(t)(\tilde{X}_t).
    \end{aligned}\label{eq:second expansion of drift of K in mu}\end{equation}
    Plugging \eqref{eq:second expansion of drift of K in x} and \eqref{eq:second expansion of drift of K in mu} into \eqref{eq:first expansion of drift of K} and the resulting terms into \eqref{eq:Drift of K} and using Schwarz' theorem to get $\partial_\mu\partial_x=\partial_x\partial_\mu$ (cf. \cite{CarmonaDelarue2018II} Remark 4.16), results in
    \begin{align}
        &K^{(1)}_t\nonumber\\
        &=\int_0^1 A_x(\theta^{\lambda,\epsilon}_t)[K^\epsilon_t]
        +\tilde{\E}[A_\mu(\theta^{\lambda,\epsilon}_t)(\tilde{X}^{\lambda,\epsilon}_t)[\tilde{K}^\epsilon_t]]\dlambda \label{eq:K:Gronwall terms}\\
        &+ \delta A_x(t) Z^\epsilon_t
        +\tilde{\E}\left[\delta A_\mu(t)(\tilde{X}_t) \tilde{Z}^\epsilon_t\right]\label{eq:K:delta terms}\\
        &+\int_0^1\lambda\int_0^1 A_{xx}(\theta^{\lambda\gamma,\epsilon}_t)[K^\epsilon_t,Y^\epsilon_t+Z^\epsilon_t]\label{eq:K:Higher order terms 1.1}\\
        &+\tilde{\E}\bigg[\tilde{\tilde{\E}}\left[A_{\mu\mu}(\theta^{\lambda\gamma,\epsilon}_t)(\tilde{X}^{\lambda\gamma,\epsilon}_t,\tilde{\tilde{X}}^{\lambda\gamma,\epsilon}_t)[\tilde{\tilde{K}}^\epsilon_t,\tilde{Y}^\epsilon_t+\tilde{Z}^\epsilon_t]\right]\label{eq:K:Higher order terms 1.2}\\
        &+2A_{x\mu}(\theta^{\lambda\gamma,\epsilon}_t)(\tilde{X}^{\lambda\gamma,\epsilon}_t)[\tilde{K}^\epsilon_t,Y^\epsilon_t+Z^\epsilon_t]
        +A_{y\mu}(\theta^{\lambda\gamma,\epsilon}_t)(\tilde{X}^{\lambda\gamma,\epsilon}_t)[\tilde{K}^\epsilon_t,\tilde{Y}^\epsilon_t+\tilde{Z}^\epsilon_t]\bigg]\dgamma\dlambda \label{eq:K:Higher order terms 2}\\
        &+\int_0^1\lambda\int_0^1  A_{xx}(\theta^{\lambda\gamma,\epsilon}_t)[Y^\epsilon_t+Z^\epsilon_t,Y^\epsilon_t+Z^\epsilon_t]
        -A_{xx}(\theta_t)[Y^\epsilon_t,Y^\epsilon_t]\label{eq:K:Y^2 terms 1}\\
        &+2\tilde{\E}\left[A_{x\mu}(\theta^{\lambda\gamma,\epsilon}_t)(\tilde{X}^{\lambda\gamma,\epsilon}_t)[\tilde{Y}^\epsilon_t+\tilde{Z}^\epsilon_t,Y^\epsilon_t+Z^\epsilon_t]
        -A_{x\mu}(\theta_t)(\tilde{X}_t)[\tilde{Y}^\epsilon_t,Y^\epsilon_t]\right]\label{eq:K:Y^2 terms 2}\\
        &+\tilde{\E}\Big[\tilde{\tilde{\E}}\big[
        A_{\mu\mu}(\theta^{\lambda\gamma,\epsilon}_t)(\tilde{X}^{\lambda\gamma,\epsilon}_t,\tilde{\tilde{X}}^{\lambda\gamma,\epsilon}_t)[\tilde{\tilde{Y}}^\epsilon_t+\tilde{\tilde{Z}}^\epsilon_t,\tilde{Y}^\epsilon_t+\tilde{Z}^\epsilon_t]
        -A_{\mu\mu}(\theta_t)(\tilde{X}_t,\tilde{\tilde{X}}_t)[\tilde{Y}^\epsilon_t,\tilde{\tilde{Y}}^\epsilon_t]\big]\Big]\label{eq:K:Y^2 terms 3}\\
        &+\tilde{\E}\big[ A_{y\mu}(\theta^{\lambda\gamma,\epsilon}_t)(\tilde{X}^{\lambda\gamma,\epsilon}_t)[\tilde{Y}^\epsilon_t+\tilde{Z}^\epsilon_t,\tilde{Y}^\epsilon_t+\tilde{Z}^\epsilon_t]
        -A_{y\mu}(\theta_t)(\tilde{X}_t)[\tilde{Y}^\epsilon_t,\tilde{Y}^\epsilon_t]\big]\dgamma\dlambda \label{eq:K:Y^2 terms 4}.
    \end{align}
    For $K^{(2)}$ the corresponding formula with $A$ replaced by $B$ holds. Now, $K$ can be estimated. \eqref{eq:K:Gronwall terms} is treated by a Gronwall argument. \eqref{eq:K:delta terms} is of high enough order by Lemma \ref{lemma:Estimates on Y and Z}. \eqref{eq:K:Higher order terms 1.1},\eqref{eq:K:Higher order terms 1.2} and \eqref{eq:K:Higher order terms 2} are of high enough order by \eqref{eq:Apriori Estimate on K} and Lemma \ref{lemma:Estimates on Y and Z}. For \eqref{eq:K:Y^2 terms 1} the terms containing $Z^\epsilon$ are again of higher order by Lemma \ref{lemma:Estimates on Y and Z}, while for the terms containing $[Y^\epsilon,Y^\epsilon]$ the difference gives the right order by the Lipschitz assumption on $A$ and $B$ from Assumption \ref{Assumption:C^{2,1}} and the estimate \eqref{eq:DeltaXEstimate}. \eqref{eq:K:Y^2 terms 2}, \eqref{eq:K:Y^2 terms 3} and \eqref{eq:K:Y^2 terms 4} are argued in the same way. For further details one can look into the proof of Lemma 5.2 from \cite{Buckdahn2016}, keeping in mind that we do not have any a priori higher order estimates, so we have to treat every second derivative term in the way the $\partial_{xx}$-terms are treated therein.
\end{proof}

\section{The Adjoint Equations}
The first-order adjoint state $(p,q)$ is defined as a solution to
\begin{equation}
    \begin{aligned}
        dp_t =& - \Big(A^*_x(\theta_t) p_t + B^*_x(\theta_t) q_t +f_x(\theta_t)\\
        &+ \tilde{\E}\left[ A^*_\mu(\tilde{\theta}_t)(X_t) \tilde{p}_t 
        + B^*_\mu(\tilde{\theta}_t)(X_t) \tilde{q}_t +f_\mu(\tilde{\theta_t})(X_t)\right]\Big) \dt 
         + q_t \dW_t, \\
        p_T =& g_x(X_T,\mu_T) + \tilde{\E}[g_\mu(\tilde{X}_T,\mu_T)(X_T)]. 
    \end{aligned}\label{eq:First order adjoint}
\end{equation}
Now, let $(e_n)_{n=1}^d$ be basis of $\R^d$ and note that $B_x(\theta_t):\R^d\to \R^{d\times d}$ linearly, i.e. $B_x(\theta_t):\R^d\times \R^d\to \R^d, (y,w) \mapsto B_x(\theta_t)[y,w]$ bilinearly, i.e. $B_x(\theta_t)[\cdot,w]\in \R^{d\times d}$. Denote $B_x^k(t):=B_x(\theta_t)[\cdot,e_k]$ and $Q^k_t:= Q_t e_k \in \R^{d\times d}$. The first-level second-order adjoint state $(P,Q)$ is defined as a solution to
\begin{equation}
    \begin{aligned}
        dP_t =& -\Big(A^*_x(\theta_t) P_t +P_t A_x(\theta_t)
        +\sum_{n}B^k_x(\theta_t)^* P_tB^k_x(\theta_t)
        + \sum_{k}B^k_x(\theta_t)^* Q^k_t \\
        &+ \sum_{n}Q^k_t B^k_x(\theta_t)
        +H_{xx}(\theta_t)+ \tilde{\E}[ H_{y\mu}(\tilde{\theta}_t)(X_t) ] \Big)\dt
        + Q_t \dW_t, \\
        P_T =& g_{xx}(X_T,\mu_T) + \tilde{\E}[g_{y\mu}(\tilde{X}_T,\mu_T)(X_T)]
    \end{aligned}
    \label{eq:First-Level second-order adjoint equation}
\end{equation}
where $H(x, \mu, \alpha, p, q):=A(x, \mu, \alpha) p+ B(x, \mu, \alpha) q-f(x, \mu, \alpha)$.
\begin{remark}
    \begin{enumerate}[(i)]
        \item Note that our first-level second-order adjoint equation differs from the second adjoint equation (3.13) in \cite{Buckdahn2016}. The analogs of the terms $\tilde{\E}[A_\mu(\tilde{\theta}_t)(X_t)]$ and $\tilde{\E}[B_\mu(\tilde{\theta}_t)(X_t)]$ appearing therein will be contained in the second-level second-order adjoint equation \eqref{eq:Second-Level Second-Order Adjoint}, as they require the setup given below.
        \item From this definition, $P_t$ is not necessarily a symmetric matrix. This can be achieved, though, by symmetrizing its equation by replacing
        \begin{align*}
            \tilde{\E}[ H_{y\mu}(\tilde{\theta}_t)(X_t) ]
            \quad&\text{by}\quad 
            \frac{1}{2}\left(\tilde{\E}[ H_{y\mu}(\tilde{\theta}_t)(X_t) ]+\tilde{\E}[ H^*_{y\mu}(\tilde{\theta}_t)(X_t) ]\right)\quad \text{and}\\
            \tilde{\E}[g_{y\mu}(\tilde{X}_T,\mu_T)(X_T)]
            \quad&\text{by}\quad 
            \frac{1}{2}\left(\tilde{\E}[g_{y\mu}(\tilde{X}_T,\mu_T)(X_T)]+\tilde{\E}[g^*_{y\mu}(\tilde{X}_T,\mu_T)(X_T)]\right).
        \end{align*}
        We omit this to avoid even heavier notation.
    \end{enumerate}
\end{remark}

\begin{remark}[Existence and uniqueness of solutions]
    \eqref{eq:First order adjoint} is a McKean-Vlasov BSDE (note that the law dependence is again on $\cL(X,p)$ not $\cL(p)$ alone) and has a unique adapted solution by already known results \cite{CarmonaDelarue2018I}. In comparison \eqref{eq:First-Level second-order adjoint equation} is not of McKean-Vlasov type but a standard BSDE and again has a unique adapted solution by already known results \cite{Pardoux2014}.
\end{remark}
\subsection{The Lifting and the Second-Level Second-Order Adjoint Equation}
    So far, we have worked on $(\Omega, \cF, (\cF_t)_{t\geq 0 }, \P)$ and used another probability space $\tilde{\Omega}$ to generate independent copies that where immediately integrated out, thereby introducing distribution dependence. Here, it was sufficient to work with an abstract probability space that was not further specified. Now, we want to linearize the interaction of two independent copies of $Y^\epsilon$. To be able to do this, we need two independent copies that are not immediately integrated out. Thus, we really need to expand our probability space. We make this explicit to avoid all possible confusion.\\
    We denote $(\hat{\Omega}, \hat{\cF}, (\hat{\cF}_t)_{t\geq 0 }, \hat{\P}):=(\Omega, \cF, (\cF_t)_{t\geq 0 }, \P)$ a copy of the original space and consider
    \begin{equation*}
        (\bar{\Omega}, \bar{\P})
        =(\Omega,  \P)\otimes  (\hat{\Omega},  \hat{\P}),
    \end{equation*}
    equipped with the completed versions $\bar{\cF}, (\bar{\cF_t})_{t\geq 0 }$ of the product $\sigma$-algebras $(\cF, (\cF_t)_{t\geq 0 })\otimes (\hat{\cF}, (\hat{\cF}_t)_{t\geq 0 })$. We will now lift all our previously defined processes to the product space by defining (with an abuse of notation) for all $\bar{\omega}\in\bar{\Omega}$
    \begin{equation*}
        W_t(\bar{\omega}):=W_t(\bar{\omega}_1),\quad X_t(\bar{\omega}):=X_t(\bar{\omega}_1)\quad\dots
    \end{equation*}
    and so forth. The resulting expectation $\bar{\E}$ then has the same effect on the lifting as the original has on the original processes, i.e. $\bar{\E}[X_t]=\E[X_t]$. This lifting allows us to easily obtain independent copies of our process $\hat{W}(\bar{\omega}):=W(\bar{\omega}_2)$, $
    \hat{X}(\bar{\omega}):=X(\bar{\omega}_2),\dots$
    For this reason we will denote $\E[\hat{\E}[\varphi(X_t,\hat{X_t})]]:=\bar{\E}[\varphi(X_t,\hat{X}_t)]$ and use the fact that $\bar{\E}[\varphi(X_t,\hat{X}_t)]=\E[\tilde{\E}[\varphi(X_t,\tilde{X}_t)]]$, where $\tilde{\E}$ denotes the previous independent lifting.\\
    We can now define the second-level second-order adjoint process $(\mathfrak{P},\mathfrak{Q}^{(1)},\mathfrak{Q}^{(2)})$, taking values in $\R^{d\times d}\times\R^{d\times d\times d}\times\R^{d\times d\times d}$, as a solution to
    \begin{equation}
        \begin{aligned}
            \mathfrak{P}_t
            &=-\bigg( 
            A^*_x(\theta_t)\mathfrak{P}_t
            +\tilde{\E}\left[A^*_\mu(\tilde{\theta}_t)(X_t)\tilde{\mathfrak{P}}_t\right]
            +\mathfrak{P}_t A_x(\hat{\theta}_t)
            +\tilde{\E}\left[\tilde{\hat{\mathfrak{P}}}_t A_\mu(\tilde{\theta}_t)(\hat{X}_t)\right]\\
            &+\tilde{\E} \left[H_{\mu\mu}(\tilde{\theta}_t)(X_t,\hat{X}_t)\right]
            +2 H_{x\mu}(\hat{\theta}_t)(X_t)
            + P^* _t A_\mu(\theta_t)(\hat{X}_t)
            + P_t A_\mu(\theta_t)(\hat{X}_t)\\
            & + B^*_x(\theta_t) P_t B_\mu(\theta_t)(\hat{X}_t) 
            + \tilde{\E}[B^*_\mu(\tilde{\theta}_t)(X_t)\tilde{P}_t B_\mu(\tilde{\theta}_t)(\hat{X}_t)]
            + B^*_x(\theta_t) P^*_t B_\mu(\theta_t)(\hat{X}_t) \\
            &+ \sum_{k=1}^d \Big(
            Q^{k,*}_t B^k_\mu(\theta_t)(\hat{X}_t)
            +Q^{k}_t B^k_\mu(\theta_t)(\hat{X}_t)
            +B^{k,*}_x(\theta_t)\mathfrak{Q}^{(1),k}_t\\
            &+\tilde{\E}[B^{k,*}_\mu(\tilde{\theta}_t)(X_t)\tilde{\mathfrak{Q}}^{(1),k}_t]
            +\mathfrak{Q}^{(2),k}_tB^k_x(\hat{\theta}_t)
            +\tilde{\E}[ \tilde{\hat{\mathfrak{Q}}}^{(2),k}_t B^k_\mu(\tilde{\theta}_t)(\hat{X}_t)]\Big)\bigg)\dt\\
            &+ \mathfrak{Q}^{(1)}_t \dW_t+\mathfrak{Q}^{(2)}_t \,\mathrm{d}\hat{W}_t,\\
            \mathfrak{P}_T&=\tilde{\E}[g_{\mu \mu}(\tilde{X}_T,\mu_T)( X_T,\hat{X}_T)] +2g_{x\mu}(\hat{X}_T,\mu_T)(X_T),
        \end{aligned}\label{eq:Second-Level Second-Order Adjoint}
    \end{equation}
    where we define
    \begin{equation*}
        \tilde{\mathfrak{P}}:\tilde{\Omega}\times \bar{\Omega}\to \R^{d\times d}, (\tilde{\omega},\bar{\omega})\mapsto \mathfrak{P}(\tilde{\omega},\bar{\omega}_2)\quad\text{and}\quad
        \tilde{\hat{\mathfrak{P}}}:\tilde{\Omega}\times \bar{\Omega}\to \R^{d\times d}, (\tilde{\omega},\bar{\omega})\mapsto \mathfrak{P}(\bar{\omega}_1,\tilde{\omega}),
    \end{equation*}
    i.e. one factor of the product space is integrated out in the equation, and similarly for $\mathfrak{Q}^{(1)}$ and $\mathfrak{Q}^{(2)}$. Therefore, this equation is of conditional McKean-Vlasov type (cf. Lemma 2.4 from \cite{CarmonaDelarue2018II}). Note that the measurability of the coefficients is clear from Fubini.
    \begin{remark}[Symmetrization of $\mathfrak{P}$]
        Since the first-level second-order adjoint can be made a symmetric-matrix-valued process, we are interested in the symmetry properties of second-level second-order adjoint. We make some adjustments to the equation:
        \begin{equation}
        \begin{aligned}
            \mathfrak{P}_t
            &=-\bigg( 
            A^*_x(\theta_t)\mathfrak{P}_t
            +\tilde{\E}\left[A^*_\mu(\tilde{\theta}_t)(X_t)\tilde{\mathfrak{P}}_t\right]
            +\mathfrak{P}_t A_x(\hat{\theta}_t)
            +\tilde{\E}\left[\tilde{\hat{\mathfrak{P}}}_t A_\mu(\tilde{\theta}_t)(\hat{X}_t)\right]\\
            &+\frac{1}{2}\tilde{\E} \left[H_{\mu\mu}(\tilde{\theta}_t)(X_t,\hat{X}_t)
            +H^*_{\mu\mu}(\tilde{\theta}_t)(\hat{X}_t,X_t)\right]
            +H_{x\mu}(\hat{\theta}_t)(X_t)
            +H^*_{x\mu}(\hat{\theta}_t)(\hat{X}_t)\\
            &+\frac{1}{2}\left( P^* _t A_\mu(\theta_t)(\hat{X}_t)
            + P_t A_\mu(\theta_t)(\hat{X}_t)+A^*_\mu(\hat{\theta}_t)(X_t)\hat{P}_t 
            + A^*_\mu(\hat{\theta}_t)(X_t)\hat{P}_t \right)\\
            & + \frac{1}{2}\Big(B^*_x(\theta_t) P_t B_\mu(\theta_t)(\hat{X}_t) 
            + \tilde{\E}[B^*_\mu(\tilde{\theta}_t)(X_t)\tilde{P}_t B_\mu(\tilde{\theta}_t)(\hat{X}_t)]\\
            &+ B^*_x(\theta_t) P^*_t B_\mu(\theta_t)(\hat{X}_t)
            + B^*_\mu(\hat{\theta}_t)(X_t)\hat{P}^*_t B_x(\hat{\theta}_t) 
            + B_\mu(\hat{\theta}_t)(X_t) \hat{P}_t B_x(\hat{\theta}_t) \Big)\\
            &+ \tilde{\E}[B^*_\mu(\tilde{\theta}_t)(X_t)\tilde{P}_t B_\mu(\tilde{\theta}_t)(\hat{X}_t)]
            + \frac{1}{2}\sum_{k=1}^d \Big(
            Q^{k,*}_t B^k_\mu(\theta_t)(\hat{X}_t)
            +Q^{k}_t B^k_\mu(\theta_t)(\hat{X}_t)\\
            &+B^{k,*}_\mu(\hat{\theta}_t)(X_t)\hat{Q}^{k}_t
            +B^{k,*}_\mu(\hat{\theta}_t)(X_t)\hat{Q}^{k,*}_t 
            +B^{k,*}_x(\theta_t)\mathfrak{Q}^{(1),k}_t
            +\mathfrak{Q}^{(2),k}_tB^k_x(\hat{\theta}_t)\\
            &+\tilde{\E}[B^{k,*}_\mu(\tilde{\theta}_t)(X_t)\tilde{\mathfrak{Q}}^{(1),k}_t
            + \tilde{\hat{\mathfrak{Q}}}^{(2),k}_t B^k_\mu(\tilde{\theta}_t)(\hat{X}_t)]\Big)\bigg)\dt
            + \mathfrak{Q}^{(1)}_t \dW_t+\mathfrak{Q}^{(2)}_t \,\mathrm{d}\hat{W}_t,\\
            \mathfrak{P}_T&=
            \frac{1}{2}
            \tilde{\E}[g_{\mu \mu}(\tilde{X}_T,\mu_T)(X_T,\hat{X}_T)
            +g^*_{\mu \mu}(\tilde{X}_T,\mu_T)(\hat{X}_T,X_T)] \\
            &+g_{x\mu}(\hat{X}_T,\mu_T)(X_T)
            +g^*_{x\mu}(X_T,\mu_T)(\hat{X}_T).
        \end{aligned}
    \end{equation}
    This process can still be used in the same dualization (seen later in \eqref{eq:second second duality}), but here, $(\mathfrak{P},\mathfrak{Q}^{(1)},\mathfrak{Q}^{(2)})$ is a symmetrization of \eqref{eq:Second-Level Second-Order Adjoint} in the sense that
    \begin{align*}
        \mathfrak{P}_t(\bar{\omega}_1,\bar{\omega}_2)
        =\mathfrak{P}^*_t(\bar{\omega}_2,\bar{\omega}_1)\quad\text{and}\quad 
        \mathfrak{Q}^{(1),k}(\bar{\omega}_1,\bar{\omega}_2)
        =\mathfrak{Q}^{(2),k,*}(\bar{\omega}_2,\bar{\omega}_1). 
    \end{align*}
    This is a sensible notion of symmetry as the same holds for the tensor product
    \begin{equation*}
        (Y_t\otimes \hat{Y}_t)(\bar{\omega}_1,\bar{\omega}_2)
        = Y_t(\bar{\omega}_1)\otimes Y_t(\bar{\omega}_2)
        = \left(Y_t(\bar{\omega}_2)\otimes Y_t(\bar{\omega}_1)\right)^*
        = (Y_t\otimes \hat{Y}_t)^*(\bar{\omega}_2,\bar{\omega}_1).
    \end{equation*}
    Again, we omit this to avoid heavier notation.
    \end{remark}
    For our notion of a solution, we define the space $S^{2, d}$ of all (equivalence classes of) $(\bar{\cF}_t)_{t\in[0,T]}$-progressively measurable continuous processes $\Phi:\bar{\Omega}\times[0,T]\to \R^d$ satisfying $\mathbb{E}[\sup _{t\in[0,T]}|\Phi_t|^2]<\infty$, equipped with the norm $\|\Phi\|_{S}^2=\mathbb{E}[\sup _{0 \leq t \leq T}|X_t|^2]$ and
    the space $\Lambda^{2, d}$ of all (equivalence classes of) $(\bar{\cF}_t)_{t\in[0,T]}$-progressively measurable processes $\Psi:\bar{\Omega}\times[0,T]\to \R^d$ satisfying $\mathbb{E}[\int_0^T|\Psi_t|^2dt]<\infty $, equipped with the norm $\|\Psi\|_{\Lambda}^2=\mathbb{E}[\int_0^T|X_t|^2dt]$.
    \begin{definition}\label{definition:solution to second-level second-order adjoint}
        On the probabilistic set-up $(\bar{\Omega},\bar{\cF},(\bar{\cF}_t)_{t\in[0,T]}, \bar{\P})$ we call a solution to the conditional McKean-Vlasov BSDE \eqref{eq:Second-Level Second-Order Adjoint} on the interval $[0, T]$ a $3$-tupel $(\mathfrak{P},\mathfrak{Q}^{(1)},\mathfrak{Q}^{(2)})\in S^{2,d\times d}\times\Lambda^{2,d\times d\times d}\times\Lambda^{2,d\times d\times d}$ such that \eqref{eq:Second-Level Second-Order Adjoint} holds almost surely as an integral equation.
    \end{definition}
    \begin{theorem}
        Under Assumption \ref{Assumption:C^{2,1}}, there exists a solution to \eqref{eq:Second-Level Second-Order Adjoint} in the sense of Definition \ref{definition:solution to second-level second-order adjoint}.
    \end{theorem}
    \begin{proof}
        We we will make a classical fixed point argument on the space
        \begin{equation*}
            \mathbb{S}:=C\left([0, T], L^2\left(\bar{\Omega},\R^{d\times d}\right)\right)\times L^2([0,T]\times\bar{\Omega},\R^{d\times d\times d})\times L^2([0,T]\times\bar{\Omega},\R^{d\times d\times d}).
        \end{equation*}
        For notational convenience, define
        \begin{align*}
            F_t
            =&\tilde{\E} \left[H_{\mu\mu}(\tilde{\theta}_t)(X_t,\hat{X}_t)\right]
            +2H_{x\mu}(\hat{\theta}_t)(X_t)
            + P^* _t A_\mu(\theta_t)(\hat{X}_t)
            + P_t A_\mu(\theta_t)(\hat{X}_t)\\
            & + B^*_x(\theta_t) P_t B_\mu(\theta_t)(\hat{X}_t) 
            + \tilde{\E}[B^*_\mu(\tilde{\theta}_t)(X_t)\tilde{P}_t B_\mu(\tilde{\theta}_t)(\hat{X}_t)]
            + B^*_x(\theta_t) P^*_t B_\mu(\theta_t)(\hat{X}_t) \\
            &+ \sum_{k=1}^d \Big(
            Q^{k,*}_t B^k_\mu(\theta_t)(\hat{X}_t)
            +Q^{k}_t B^k_\mu(\theta_t)(\hat{X}_t)
        \end{align*}
        which by our assumptions on the coefficients is an element of $\Lambda^{2,d\times d}$. Now, fix $(\Phi,\Psi^{(1)},\Psi^{(2)})\in \mathbb{S}$ and consider the standard BSDE
        \begin{equation}
            \begin{aligned}
                \mathfrak{P}_t
                =&-\Big( 
                A^*_x(\theta_t)\mathfrak{P}_t
                +\mathfrak{P}_t A_x(\hat{\theta}_t)
                +\tilde{\E}\left[A^*_\mu(\tilde{\theta}_t)(X_t)\tilde{\Phi}_t\right]
                +\tilde{\E}\left[\tilde{\hat{\Phi}}_t A_\mu(\tilde{\theta}_t)(\hat{X}_t)\right]
                +F_t\\
                &+\sum_{k=1}^d\big( B^{k,*}_x(\theta_t)\mathfrak{Q}^{(1),k}_t
                +\mathfrak{Q}^{(2),k}_t B^k_x(\hat{\theta}_t)\\
                &+\tilde{\E}[B^{k,*}_\mu(\tilde{\theta}_t)(X_t)\tilde{\Psi}^{(1),k}_t
                +\tilde{\hat{\Psi}}^{(2),k}_t B^k_\mu(\tilde{\theta}_t)(\hat{X}_t)]\big)\Big)\dt
                + \mathfrak{Q}^{(1)}_t \dW_t+\mathfrak{Q}^{(2)}_t \,\mathrm{d}\hat{W}_t,\\
                \mathfrak{P}_T=&\tilde{\E}[g_{\mu \mu}(\tilde{X}_T,\mu_T)( X_T,\hat{X}_T)] +2g_{x\mu}(\hat{X}_T,\mu_T)(X_T).
            \end{aligned}\label{eq:Second-Level Second-Order Adjoint in Fixed point argument}
        \end{equation}
        By our assumptions on the coefficients, this BSDE has a unique solution, which we will denote $(\mathfrak{P}^{\Phi,\Psi},\mathfrak{Q}^{\Phi,\Psi,(1)},\mathfrak{Q}^{\Phi,\Psi,(2)})\in S^{2,d}\times\Lambda^{2,d\times d}\times\Lambda^{2,d\times d}\subset \mathbb{S}$ (cf. \cite{Pardoux2014} Theorem 5.17). Notice that the original filtration $(\cF_t)_{t\geq 0 }$ was generated by $(W_t)_{t\geq 0 }$. Therefore $(\bar{\cF}_t)_{t\geq 0 }$ coincides with the completed filtration generated by $(W_t,\hat{W_t})_{t\geq 0}$. Now, consider the map
        \begin{equation*}
            \Xi: \mathbb{S}\to \mathbb{S},\quad 
            (\Phi,\Psi^{(1)},\Psi^{(2)})\mapsto 
            (\mathfrak{P}^{\Phi,\Psi},\mathfrak{Q}^{\Phi,\Psi,(1)},\mathfrak{Q}^{\Phi,\Psi,(2)}).
        \end{equation*}
        We will show that $\Xi$ is a contraction with respect to the metric 
        \begin{align*}
            &\rho((\Phi,\Psi^{(1)},\Psi^{(2)}),(\phi,\psi^{(1)},\psi^{(2)}))\\
            &:=\sup _{t \in[0, T]} e^{\kappa t} \bar{\E}\left[\left\|\Phi_t- \phi_t\right\|^2\right]
            +\frac{3}{4}\bar{\E}\left[\int_0^T e^{\kappa t} \left\| \Psi^{(1)}_t-\psi^{(1)}_t\right\|^2+e^{\kappa t} \left\| \Psi^{(2)}_t-\psi^{(2)}_t\right\|^2\dt\right],
        \end{align*}
        where $\kappa>0$ will be chosen later, which gives us a unique solution to our desired equation. So let $(\Phi,\Psi^{(1)},\Psi^{(2)}),(\phi,\psi^{(1)},\psi^{(2)})\in \mathbb{S}$ and denote 
        $\Delta^\mathfrak{P}=\mathfrak{P}^{\Phi,\Psi}-\mathfrak{P}^{\phi,\psi}$, 
        $\Delta^{\mathfrak{Q}^{(1)}}=\mathfrak{Q}^{\Phi,\Psi,(1)}-\mathfrak{Q}^{\phi,\psi,(1)}$ and 
        $\Delta^{\mathfrak{Q}^{(2)}}=\mathfrak{Q}^{\Phi,\Psi,(2)}-\mathfrak{Q}^{\phi,\psi,(2)}$ and similarly for $\Delta^\Phi,\Delta^{\Psi^{(1)}},\Delta^{\Psi^{(2)}}$. By Ito's formula
        \begin{align*}
            0=&\bar{\E}[e^{\kappa T}\|\Delta^\mathfrak{P}_T\|^2]\\
            =&\bar{\E}\Big[e^{\kappa s}\|\Delta^\mathfrak{P}_s\|^2
            -\int_s^T e^{\kappa t}2\langle \Delta^\mathfrak{P}_t,A^*_x(\theta_t) 
            \Delta^\mathfrak{P}_t 
            +\Delta^\mathfrak{P}_t A_x(\hat{\theta}_t)
            +\tilde{\E}[A^*_\mu(\tilde{\theta}_t)(X_t)\tilde{\Delta}^\Phi_t\\
            &+\tilde{\hat{\Delta}}^{\Phi}_t A_\mu(\tilde{\theta}_t)(\hat{X}_t)]
            +\sum_{k=1}^d B^{k,*}_x(\theta_t)\Delta^{\mathfrak{Q}^{(1)},k}_t 
            +\Delta^{\mathfrak{Q}^{(2)},k}_tB^k_x(\hat{\theta}_t)
            +\tilde{\E}[B^{k,*}_\mu(\tilde{\theta}_t)(X_t)\tilde{\Delta}^{\Psi^{(1),k}}_t\\
            &+\tilde{\hat{\Delta}}^{\Psi^{(2)},k}_t B^k_\mu(\tilde{\theta}_t)(\hat{X}_t)]\rangle
            +e^{\kappa t}\|\Delta^{\mathfrak{Q}^{(1)}}_t\|^2 
            + e^{\kappa t}\|\Delta^{\mathfrak{Q}^{(2)}}_t\|^2 
            +\kappa e^{\kappa t}\|\Delta^{\mathfrak{P}}_t\|^2\dt \Big],
        \end{align*}
        so rearranging gives
        \begin{align*}
            &\bar{\E}\Big[ e^{\kappa s}\|\Delta^\mathfrak{P}_s\|^2+\int_s^T e^{\kappa t}\|\Delta^{\mathfrak{Q}^{(1)}}_t\|^2 \dt
            +\int_s^T e^{\kappa t}\|\Delta^{\mathfrak{Q}^{(2)}}_t\|^2 \dt 
            +\kappa\int_s^T e^{\kappa t}\|\Delta^{\mathfrak{P}}_t\|^2\dt \Big]\\
            =&\bar{\E}\Big[\int_s^T 2e^{\kappa t}\langle \Delta^\mathfrak{P}_t,A^*_x(\theta_t) 
            \Delta^\mathfrak{P}_t 
            +\Delta^\mathfrak{P}_t A_x(\hat{\theta}_t)
            +\tilde{\E}[A^*_\mu(\tilde{\theta}_t)(X_t)\tilde{\Delta}^\Phi_t
            +\tilde{\hat{\Delta}}^{\Phi}_t A_\mu(\tilde{\theta}_t)(\hat{X}_t)]\\
            &+\sum_{k=1}^d B^{k,*}_x(\theta_t)\Delta^{\mathfrak{Q}^{(1)},k}_t 
            +\Delta^{\mathfrak{Q}^{(2)},k}_tB^k_x(\hat{\theta}_t)
            +\tilde{\E}[B^{k,*}_\mu(\tilde{\theta}_t)(X_t)\tilde{\Delta}^{\Psi^{(1)},k}_t\\
            &+\tilde{\hat{\Delta}}^{\Psi^{(2)},k}_t B^k_\mu(\tilde{\theta}_t)(\hat{X}_t)]\rangle \dt\Big].
        \end{align*}
        Now, using the boundedness of the coefficients from Assumption \ref{Assumption:C^{2,1}}, taking the supremum over $s\in[0,T]$ and using the Young-inequality multiple times, we arrive at
        \begin{align*}
            &\sup_{t\in[0,T]}\bar{\E}\left[e^{\kappa t}\|\Delta^\mathfrak{P}_t\|^2\right]+\bar{\E}\left[\int_0^T e^{\kappa t}\|\Delta^{\mathfrak{Q}^{(1)}}_t\|^2 
            + e^{\kappa t}\|\Delta^{\mathfrak{Q}^{(2)}}_t\|^2 \dt 
            +\kappa\int_0^T e^{\kappa t}\|\Delta^{\mathfrak{P}}_t\|^2\dt\right]\\
            \leq& \bar{\E}\Bigg[c_1 \int_0^T e^{\kappa t}\|\Delta^\mathfrak{P}_t\|\bigg(\| 
            \Delta^\mathfrak{P}_t\| 
            +\tilde{\E}[\|\tilde{\Delta}^\Phi_t\|
            +\|\tilde{\hat{\Delta}}^{\Phi}_t\|]\\
            &+\sum_{k=1}^d \left(\|\Delta^{\mathfrak{Q}^{(1)},k}_t \|
            +\|\Delta^{\mathfrak{Q}^{(2)},k}_t\|
            +\tilde{\E}[\|\tilde{\Delta}^{\Psi^{(1),k}}_t\| 
            + \|\tilde{\hat{\Delta}}^{\Psi^{(2)},k}_t\|]\right)\bigg) \dt\Bigg]\\
            \leq & \bar{\E}\Bigg[ \int_0^T c_2e^{\kappa t}\|\Delta^\mathfrak{P}_t\|^2+\frac{1}{3}e^{\kappa t}\| 
            \Delta^\mathfrak{P}_t\|^2 
            +\frac{1}{3}e^{\kappa t}\tilde{\E}[\|\tilde{\Delta}^\Phi_t\|^2
            +\|\tilde{\hat{\Delta}}^{\Phi}_t\|^2]\\
            &\frac{1}{4}e^{\kappa t}\|\Delta^{\mathfrak{Q}^{(1)}}_t \|^2
            +\frac{1}{4}e^{\kappa t}\|\Delta^{\mathfrak{Q}^{(2)}}_t\|^2
            +\frac{1}{4}e^{\kappa t}\tilde{\E}[\|\tilde{\Delta}^{\Psi^{(1)}}_t\|^2
            +\|\tilde{\hat{\Delta}}^{\Psi^{(2)}}_t\|^2]\bigg) \dt\Bigg].
        \end{align*}
        By virtue of the Young inequality $c_2$ is chosen accordingly and depends on $c_1$ and therefore on the constants coming from the coefficients. Finally, taking $\kappa=c_2+\frac{1}{3}$ gives the result with contraction coefficient $\frac{1}{3}$.
    \end{proof}
\section{The Duality Relations}
Having in mind the desired linearization of terms in the Taylor-expansion of the cost functional, we will need the following duality relations, which arise directly from Ito's formula. First, we have the dualization of the first-order adjoint with the first-order variational process
\begin{equation}
    \begin{aligned}
        \bar{\E}\left[\langle p_T, Y^\epsilon_T\rangle\right]
        =& - \bar{\E}\Big[\int_0^T \langle f_x(\theta_t), Y^\epsilon_t\rangle 
        + \tilde{\E}\left[\langle f_\mu(\tilde{\theta}_t)(X_t), Y^\epsilon_t\rangle\right]\\
        &-\langle p_t,\delta A(t) \rangle 
        - \langle q_t,\delta B(t) \rangle \dt\Big],
    \end{aligned} \label{eq:duality p and Y}
\end{equation}
and the dualization of the first-order adjoint with the second-order variational process
\begin{equation}
    \begin{aligned}
        &\bar{\E}\left[\langle p_T, Z^\epsilon_T\rangle\right]\\
        =& \bar{\E}\bigg[\int_0^T \langle -f_x(\theta_t), Z^\epsilon_t\rangle
        - \tilde{\E}[\langle f_\mu(\tilde{\theta}_t)(X_t), Z^\epsilon_t\rangle]  \\
        & + \big\langle p_t,  \frac{1}{2} A_{xx}(\theta_t)[Y^\epsilon_t,Y^\epsilon_t]
        +  \frac{1}{2} \tilde{\E}\left[ A_{y\mu}(\theta_t)(\tilde X_t)[ \tilde Y^\epsilon_t,\tilde Y^\epsilon_t]\right]\\ 
        &+\frac{1}{2}\tilde{\E} \Big[\tilde{\tilde\E}\big[ A_{\mu\mu}(\theta_t)(\tilde{X}_t,\tilde{\tilde{X}}_t)[\tilde{Y}^\epsilon_t,\tilde{\tilde{Y}}^\epsilon_t]\big]\Big] 
        +\tilde{\E}\left[ A_{x\mu}(\theta_t)(\tilde X_t)[ \tilde Y^\epsilon_t,Y^\epsilon_t]\right]
        \big\rangle \\
        &+ \big\langle q_t,  \frac{1}{2} B_{xx}(\theta_t)[Y^\epsilon_t,Y^\epsilon_t]
        +  \frac{1}{2} \tilde{\E}\left[ B_{y\mu}(\theta_t)(\tilde X_t)[ \tilde Y^\epsilon_t,\tilde Y^\epsilon_t]\right] \\
        &+\frac{1}{2}\tilde{\E} \Big[\tilde{\tilde\E}\big[ B_{\mu\mu}(\theta_t)(\tilde{X}_t,\tilde{\tilde{X}}_t)[\tilde{Y}^\epsilon_t,\tilde{\tilde{Y}}^\epsilon_t]\big]\Big] 
        +\tilde{\E}\left[ B_{x\mu}(\theta_t)(\tilde X_t)[ \tilde Y^\epsilon_t,Y^\epsilon_t]\right]
        \big\rangle  \dt\bigg] 
        +o(\epsilon).
    \end{aligned} \label{eq:duality p and Z}
\end{equation}
In preparation for the dualization of the first-level second-order adjoint, we look at the tensor product of the first-order variational process $Y^\epsilon$ with itself. We use the tensor notation but note that, for $x,y\in\R^d$, it holds $x\otimes y:=xy^\top\in \R^{d\times d}$, since we work in finite dimensions. We get
\begin{align*}
    &d\left(Y^\epsilon_t \otimes Y^\epsilon_t\right)\\
    &=\big((A_x(\theta_t)Y^\epsilon_t+\tilde{\E}[A_\mu(\theta_t)(\tilde{X}_t)\tilde{Y}^\epsilon_t]+\delta A(t)) \otimes Y^\epsilon_t\\
    &+Y^\epsilon_t \otimes (A_x(\theta_t)Y^\epsilon_t+\tilde{\E}[A_\mu(\theta_t)(\tilde{X}_t)\tilde{Y}^\epsilon_t]+\delta A(t))\\
    &+(B_x(\theta_t)Y^\epsilon_t+\tilde{\E}[B_\mu(\theta_t)(\tilde{X}_t)\tilde{Y}^\epsilon_t]+\delta B(t))(B_x(\theta_t)Y^\epsilon_t+\tilde{\E}[B_\mu(\theta_t)(\tilde{X}_t)\tilde{Y}^\epsilon_t]+\delta B(t))^*\big) \dt\\
    &+\sum_{k=1}^d\Big(\big(B^k_x(\theta_t)Y^\epsilon_t+\tilde{\E}[B^k_\mu(\theta_t)(\tilde{X}_t)\tilde{Y}^\epsilon_t]+\delta B^k(t) \big) \otimes Y^\epsilon_t\\
    &+Y^\epsilon_t \otimes \big(B^k_x(\theta_t)Y^\epsilon_t+\tilde{\E}[B^k_\mu(\theta_t)(\tilde{X}_t)\tilde{Y}^\epsilon_t]+\delta B^k(t) \big)\Big)\dW^k_t,
\end{align*}
and therefore, for the dualization of the first-level second-order adjoint process with this tensor product
\begin{equation}
    \begin{aligned}
        &\bar{\E}\left[\left\langle P_T, Y^\epsilon_T \otimes Y^\epsilon_T\right\rangle\right]\\
        =&\bar{\E}\bigg[\int_0^T-\langle H_{xx}(\theta_t)+ \tilde{\E}[ H_{y\mu}(\tilde{\theta}_t)(X_t) ]  , Y^\epsilon_t \otimes Y^\epsilon_t\rangle
        +\langle P_t, \tilde{\E}[A_\mu(\theta_t)(\tilde{X}_t)\tilde{Y}^\epsilon_t]\otimes Y^\epsilon_t\\
        &+Y^\epsilon_t \otimes \tilde{\E}[A_\mu(\theta_t)(\tilde{X}_t)\tilde{Y}^\epsilon_t]+\delta B(t)\delta B^*(t)
        +B_x(\theta_t)Y^\epsilon_t\tilde{\E}[B_\mu(\theta_t)(\tilde{X}_t)\tilde{Y}^\epsilon_t]^*\\
        &+\tilde{\E}[B_\mu(\theta_t)(\tilde{X}_t)\tilde{Y}^\epsilon_t]\tilde{\E}[B_\mu(\theta_t)(\tilde{X}_t)\tilde{Y}^\epsilon_t]^*
        +\tilde{\E}[B_\mu(\theta_t)(\tilde{X}_t)\tilde{Y}^\epsilon_t](B_x(\theta_t)Y^\epsilon_t)^*\rangle \\
        & +\sum_{k=1}^d \langle Q^k_t, \tilde{\E}[B^k_\mu(\theta_t)(\tilde{X}_t)\tilde{Y}^\epsilon_t] \otimes Y^\epsilon_t+Y^\epsilon_t \otimes \tilde{\E}[B^k_\mu(\theta_t)(\tilde{X}_t)\tilde{Y}^\epsilon_t] \rangle \dt\bigg]+o(\epsilon).
    \end{aligned}\label{eq:first second duality}
\end{equation}
Note that for all the above dualizations it would have sufficed to look at the processes on the original probability space and use the standard expectation. For the dualization of the second-level second-order adjoint process the construction of the probability space $\bar{\Omega}$ is important, as the dualization needs to be with the tensor product of the first-order variational process $Y^\epsilon$ with its independent copy $\hat{Y}^\epsilon$ (introduced above), resulting in
\begin{align*}
    &d\left(Y^\epsilon_t \otimes \hat{Y}^\epsilon_t\right)\\
    =&\Big((A_x(\theta_t)Y^\epsilon_t+\tilde{\E}[A_\mu(\theta_t)(\tilde{X}_t)\tilde{Y}^\epsilon_t]+\delta A(t)) \otimes \hat{Y}^\epsilon_t
    +Y^\epsilon_t \otimes (A_x(\hat{\theta}_t)\hat{Y}^\epsilon_t+\tilde{\E}[A_\mu(\hat{\theta}_t)(\tilde{X}_t)\tilde{Y}^\epsilon_t]\\
    &+\delta \hat{A}(t))\Big) \dt
    +\sum_{k=1}^d\left(\left(B^k_x(\theta_t)Y^\epsilon_t+\tilde{\E}[B^k_\mu(\theta_t)(\tilde{X})\tilde{Y}^\epsilon_t]+\delta B^k(t) \right) \otimes \hat{Y}^\epsilon_t\right) \dW^k_t\\
    &+\sum_{k=1}^d\left(Y^\epsilon_t \otimes \left(B^k_x(\hat{\theta}_t)\hat{Y}^\epsilon_t+\tilde{\E}[B^k_\mu(\hat{\theta}_t)(\tilde{X}_t)\tilde{Y}^\epsilon_t]+\delta \hat{B}^k(t) \right)\right)\,\mathrm{d}\hat{W}^k_t,
\end{align*}
and thus, using Fubini,
\begin{equation}
    \begin{aligned}
        \bar{\E}\left[\left\langle \mathfrak{P}_T, Y^\epsilon_T \otimes \hat{Y}^\epsilon_T\right\rangle\right]
        &=\bar{\E}\bigg[\int_0^T-\big\langle \tilde{\E} [H_{\mu\mu}(\tilde{\theta}_t)(X_t,\hat{X}_t)]
        +2 H_{x\mu}(\hat{\theta}_t)(X_t)\\
        &+ P^* _t A_\mu(\theta_t)(\hat{X}_t)
        + P_t A_\mu(\theta_t)(\hat{X}_t)
        + B^*_x(\theta_t) P_t B_\mu(\theta_t)(\hat{X}_t)\\ 
        &+ \tilde{\E}[B^*_\mu(\tilde{\theta}_t)(X_t)\tilde{P}_t B_\mu(\tilde{\theta}_t)(\hat{X}_t)]
        + B^*_x(\theta_t) P^*_t B_\mu(\theta_t)(\hat{X}_t) \\
        &+ \sum_{k=1}^d
        Q^{k,*}_t B^k_\mu(\theta_t)(\hat{X}_t)
        +Q^{k}_t B^k_\mu(\theta_t)(\hat{X}_t), Y^\epsilon_t \otimes \hat{Y}^\epsilon_t\big\rangle \dt\bigg]+o(\epsilon),
    \end{aligned}\label{eq:second second duality}
\end{equation}
\begin{remark}
    Note the appearance of the $\langle P_t,\delta B(t) \delta B^*(t)\rangle$-term in the dualization \eqref{eq:first second duality}. This term arises from the quadratic variation of $Y^\epsilon$ and neither is of high order, nor can be eliminated by the dualization. Therefore it, and in particular the first-level second-order adjoint process, will appear in the maximum principle. On the other hand, the stochastic integrals appearing in the equations for $Y^\epsilon$ and $\hat{Y}^\epsilon$ are by construction independent, which implies that their co-variation is $0$. Thus, no such term appears in the dualization \eqref{eq:second second duality}, letting the second-level second-order adjoint process not appear in the final maximum principle. This is reminiscent of the sharper estimates for the terms containing two independent copies of the variational process from \citet{Buckdahn2016}.
\end{remark}
\section{The Expansion of the Cost Functional and the Maximum Principle}
    We come to the most important result towards the maximum principle. This result was already shown in \cite{Buckdahn2016}, merely our proof is new, as it utilizes the second-level second-order adjoint and its dualization.
    \begin{theorem}
        Under Assumption \ref{Assumption:C^{2,1}}, if $\alpha\in \A$ is optimal, it holds
        \begin{equation}
            J(\alpha^{\epsilon})-J(\alpha)=  \bar{\E} \left[ \int_0^T \delta H(t)
            +\frac{1}{2}\big\langle P_t, \delta B(t)\delta B(t)^*\big\rangle \dt\right] +o(\epsilon).\label{eq:Final Expansion of Cost Functional}
        \end{equation}
    \end{theorem}
    \begin{proof}
        Using Taylor's formula (cf. Lemma \ref{lemma:Taylor formula for Lions Derivative}), we get the expansion of the cost functional,
        \begin{align*}
            &J(\alpha^{\epsilon})-J(\alpha )\\
            =&\mathbb{E}\Big[\int_0^T f\left(\theta_t^\epsilon\right)-f\left(\theta_t\right) \dt
            +g(X_T^{\epsilon}, \mu^\epsilon_T)-g(X_T, \mu_T)\Big] \\
            = & \bar{\E} \Bigg[ \int_0^T f_x(\theta_t)[\Delta X_t]+\tilde{\mathbb{E}}\left[f_\mu(\theta_t)(\tilde{X}_t)[\Delta \tilde{X}_t\right]
            +\delta f(t)\\
            &+\frac{1}{2}f_{xx}(\theta_t)[\Delta X_t,\Delta X_t]
            +\tilde{\mathbb{E}}\left[f_{x\mu}(\theta_t)(\tilde{X}_t)[\Delta \tilde{X}_t,\Delta X_t]\right]\\
            & +\frac{1}{2}\tilde{\mathbb{E}}\left[f_{y\mu}(\theta_t)(\tilde{X}_t)[\Delta \tilde{X}_t,\Delta \tilde{X}_t]\right]
            +\frac{1}{2}\tilde{\tilde{\E}} \left[\tilde{\mathbb{E}} \left[f_{\mu\mu}(\theta_t)(\tilde{X}_t,\tilde{\tilde{X}}_t)[\Delta \tilde{X}_t,\Delta \tilde{\tilde{X}}_t]\right]\right] \dt\\
            & +g_x(X_T,\mu_T)[\Delta X_T]
            +\tilde{\mathbb{E}}\left[g_\mu(X_T,\mu_T)(\tilde{X}_T)[\Delta \tilde{X}_T]\right] \\
            & +\frac{1}{2}g_{x x}(X_T,\mu_T)[\Delta X_T,\Delta X_T]
            +\frac{1}{2}\tilde{\mathbb{E}}\left[g_{y\mu}(X_T,\mu_T)(\tilde{X}_T)[\Delta \tilde{X}_T,\Delta \tilde{X}_T]\right]\\
            & +\tilde{\mathbb{E}}\left[g_{x\mu}(X_T,\mu_T)(\tilde{X}_T)[\Delta \tilde{X}_T,\Delta X_T]\right]\\
            &+\frac{1}{2}\tilde{\tilde{\E}} \left[\tilde{\mathbb{E}} \left[g_{\mu\mu}(X_T,\mu_T)(\tilde{X}_T,\tilde{\tilde{X}}_T)[\Delta \tilde{X}_T,\Delta \tilde{\tilde{X}}_T]\right]\right]\Bigg]
            +o(\epsilon) .
        \end{align*}
        Now, by the linearity, we can everywhere replace $\Delta X_t$ by $Y^\epsilon_t+Z^\epsilon_t$ and the difference is in $o(\epsilon)$ by the estimate from Lemma \ref{lemma:estimate on DeltaX-Y-Z}. Also the terms containing $[Y^\epsilon ,Z^\epsilon]$ or $[Z^\epsilon,Z^\epsilon]$ (the quadratic terms from the second derivatives) are of high-enough order by the estimates from Lemma \ref{lemma:Estimates on Y and Z}, so that 
        \begin{equation*}
            \E\left[\frac{1}{2}f_{xx}(\theta_t)[Y^\epsilon_t+Z^\epsilon_t,Y^\epsilon_t+Z^\epsilon_t]\right]
            =\E\left[\frac{1}{2}f_{xx}(\theta_t)[Y^\epsilon_t,Y^\epsilon_t]\right]+o(\epsilon),
        \end{equation*}
        and similarly for the other terms.
        This results in
        \begin{align}
            J(\alpha^{\epsilon})-J(\alpha )
            = & \bar{\E}\Bigg[ \int_0^T f_x(\theta_t)[Y^\epsilon_t+Z^\epsilon_t]+\tilde{\mathbb{E}}\left[f_\mu(\theta_t)(\tilde{X}_t)(\tilde{Y}^\epsilon_t+\tilde{Z}^\epsilon_t)\right]
            +\delta f(t)\nonumber\\
            &+\frac{1}{2}f_{xx}(\theta_t)[Y^\epsilon_t,Y^\epsilon_t]
            +\tilde{\mathbb{E}}\left[f_{x\mu}(\theta_t)[\tilde{X}_t)[\tilde{Y}^\epsilon_t,Y^\epsilon_t]\right]\nonumber\\
            & +\frac{1}{2}\tilde{\mathbb{E}}\left[f_{y\mu}(\theta_t)(\tilde{X}_t)[\tilde{Y}^\epsilon_t,\tilde{Y}^\epsilon_t]\right]
            +\frac{1}{2}\tilde{\tilde{\E}} \left[\tilde{\mathbb{E}} \left[f_{\mu\mu}(\theta_t)(\tilde{X}_t,\tilde{\tilde{X}}_t)[\tilde{Y}^\epsilon_t,\tilde{\tilde{Y}}^\epsilon_t]\right]\right] \dt\nonumber\\
            & +g_x(X_T,\mu_T)[Y^\epsilon_T+Z^\epsilon_T]
            +\tilde{\mathbb{E}}\left[g_\mu(X_T,\mu_T)(\tilde{X}_T)[\tilde{Y}^\epsilon_T+\tilde{Z}^\epsilon_T]\right] \label{eq:First Duality Replacement}\\
            & +\frac{1}{2}g_{x x}(X_T,\mu_T)[Y^\epsilon_T,Y^\epsilon_T]
            +\frac{1}{2}\tilde{\mathbb{E}}\left[g_{y\mu}(X_T,\mu_T)(\tilde{X}_T)[\tilde{Y}^\epsilon_T,\tilde{Y}^\epsilon_T]\right] \nonumber
            \\
            &+\tilde{\mathbb{E}}\left[g_{x\mu}(X_T,\mu_T)(\tilde{X}_T)[\tilde{Y}^\epsilon_T,Y^\epsilon_T]\right] \nonumber\\
            &+\frac{1}{2}\tilde{\tilde{\E}} \left[\tilde{\mathbb{E}} \left[g_{\mu\mu}(X_T,\mu_T)(\tilde{X}_T,\tilde{\tilde{X}}_T)[\tilde{Y}^\epsilon_T,\tilde{\tilde{Y}}^\epsilon_T]\right]\right]\Bigg]+o(\epsilon)\nonumber
        \end{align}
        Now, we can replace \eqref{eq:First Duality Replacement} with the duality relations \eqref{eq:duality p and Y} and \eqref{eq:duality p and Z} to get
        \begin{align}
            J(\alpha^{\epsilon})-J(\alpha )
            = & \bar{\E} \Bigg[ \int_0^T \delta H(t)
            +\frac{1}{2}H_{xx}(\theta_t)[Y^\epsilon_t,Y^\epsilon_t]
            +\tilde{\mathbb{E}}\left[H_{x\mu}(\theta_t)[\tilde{X}_t)[\tilde{Y}^\epsilon_t,Y^\epsilon_t]\right]\nonumber\\
            & +\frac{1}{2}\tilde{\mathbb{E}}\left[H_{y\mu}(\theta_t)(\tilde{X}_t)[\tilde{Y}^\epsilon_t,\tilde{Y}^\epsilon_t]\right]
            +\frac{1}{2}\tilde{\tilde{\E}} \left[\tilde{\mathbb{E}} \left[H_{\mu\mu}(\theta_t)(\tilde{X}_t,\tilde{\tilde{X}}_t)[\tilde{Y}^\epsilon_t,\tilde{\tilde{Y}}^\epsilon_t]\right]\right] \dt\nonumber\\
            & +\frac{1}{2}g_{x x}(X_T,\mu_T)[Y^\epsilon_T,Y^\epsilon_T]
            +\frac{1}{2}\tilde{\mathbb{E}}\left[g_{y\mu}(X_T,\mu_T)(\tilde{X}_T)[\tilde{Y}^\epsilon_T,\tilde{Y}^\epsilon_T]\right] \label{eq:Second Duality Replacement}
            \\
            &+\tilde{\mathbb{E}}\left[g_{x\mu}(X_T,\mu_T)(\tilde{X}_T)[\tilde{Y}^\epsilon_T,Y^\epsilon_T]\right] \nonumber\\
            &+\frac{1}{2}\tilde{\tilde{\E}} \left[\tilde{\mathbb{E}} \left[g_{\mu\mu}(X_T,\mu_T)(\tilde{X}_T,\tilde{\tilde{X}}_T)[\tilde{Y}^\epsilon_T,\tilde{\tilde{Y}}^\epsilon_T]\right]\right]\Bigg]
            +o(\epsilon),\nonumber
        \end{align}
        and then replace
        \eqref{eq:Second Duality Replacement} with the duality relation \eqref{eq:first second duality} to arrive at
        \begin{align}
            &J(\alpha^{\epsilon})-J(\alpha )\nonumber\\
            = & \bar{\E}\Bigg[ \int_0^T \delta H(t)
            +\tilde{\mathbb{E}}\left[H_{x\mu}(\theta_t)[\tilde{X}_t)[\tilde{Y}^\epsilon_t,Y^\epsilon_t]\right]
            +\frac{1}{2}\tilde{\tilde{\E}} \left[\tilde{\mathbb{E}} \left[H_{\mu\mu}(\theta_t)(\tilde{X}_t,\tilde{\tilde{X}}_t)[\tilde{Y}^\epsilon_t,\tilde{\tilde{Y}}^\epsilon_t]\right]\right] \nonumber\\
            &+\frac{1}{2}\big\langle P_t, \tilde{\E}[A_\mu(\theta_t)(\tilde{X}_t)\tilde{Y}^\epsilon_t]\otimes Y^\epsilon_t
            +Y^\epsilon_t \otimes \tilde{\E}[A_\mu(\theta_t)(\tilde{X}_t)\tilde{Y}^\epsilon_t]+\delta B(t)\delta B(t)^*  \nonumber\\
            &+B_x(\theta_t)Y^\epsilon_t\tilde{\E}[B_\mu(\theta_t)(\tilde{X}_t)\tilde{Y}^\epsilon_t]^*
            +\tilde{\E}[B_\mu(\theta_t)(\tilde{X}_t)\tilde{Y}^\epsilon_t]\tilde{\E}[B_\mu(\theta_t)(\tilde{X}_t)\tilde{Y}^\epsilon_t]^*  \nonumber\\
            &+\tilde{\E}[B_\mu(\theta_t)(\tilde{X}_t)\tilde{Y}^\epsilon_t](B_x(\theta_t)Y^\epsilon_t)^*\big\rangle\nonumber\\
            &+\frac{1}{2}\sum_{k=1}^d \big\langle Q^k_t, \tilde{\E}[B^k_\mu(\theta_t)(\tilde{X}_t)\tilde{Y}^\epsilon_t] \otimes Y^\epsilon_t+Y^\epsilon_t \otimes \tilde{\E}[B^k_\mu(\theta_t)(\tilde{X}_t)\tilde{Y}^\epsilon_t] \big\rangle  \dt \nonumber \\
            &+\tilde{\mathbb{E}}\left[g_{x\mu}(X_T,\mu_T)(\tilde{X}_T)[\tilde{Y}^\epsilon_T,Y^\epsilon_T]\right] 
            +\frac{1}{2}\tilde{\tilde{\E}} \left[\tilde{\mathbb{E}} \left[g_{\mu\mu}(X_T,\mu_T)(\tilde{X}_T,\tilde{\tilde{X}}_T)[\tilde{Y}^\epsilon_T,\tilde{\tilde{Y}}^\epsilon_T]\right]\right]\Bigg] \label{eq:Third Duality Replacement}
            \\
            &+o(\epsilon).\nonumber
        \end{align}
        Finally, we can now use the second-level second-order adjoint process and its dualization \eqref{eq:second second duality} to replace \eqref{eq:Third Duality Replacement} and, using Fubini, conclude that
        \begin{align*}
            &J(\alpha^{\epsilon})-J(\alpha )\\
            = & \bar{\E}\Bigg[ \int_0^T \delta H(t)
            +\tilde{\mathbb{E}}\left[H_{x\mu}(\theta_t)[\tilde{X}_t)[\tilde{Y}^\epsilon_t,Y^\epsilon_t]\right]
            +\frac{1}{2}\tilde{\tilde{\E}} \left[\tilde{\mathbb{E}} \left[H_{\mu\mu}(\theta_t)(\tilde{X}_t,\tilde{\tilde{X}}_t)[\tilde{Y}^\epsilon_t,\tilde{\tilde{Y}}^\epsilon_t]\right]\right] \\
            &+\frac{1}{2}\big\langle P_t, \tilde{\E}[A_\mu(\theta_t)(\tilde{X}_t)\tilde{Y}^\epsilon_t]\otimes Y^\epsilon_t
            +Y^\epsilon_t \otimes \tilde{\E}[A_\mu(\theta_t)(\tilde{X}_t)\tilde{Y}^\epsilon_t]+\delta B(t)\delta B(t)^*  \\
            &+B_x(\theta_t)Y^\epsilon_t\tilde{\E}[B_\mu(\theta_t)(\tilde{X}_t)\tilde{Y}^\epsilon_t]^*
            +\tilde{\E}[B_\mu(\theta_t)(\tilde{X}_t)\tilde{Y}^\epsilon_t]\tilde{\E}[B_\mu(\theta_t)(\tilde{X}_t)\tilde{Y}^\epsilon_t]^*  \\
            &+\tilde{\E}[B_\mu(\theta_t)(\tilde{X}_t)\tilde{Y}^\epsilon_t](B_x(\theta_t)Y^\epsilon_t)^*\big\rangle\nonumber\\
            &+\frac{1}{2}\sum_{k=1}^d \big\langle Q^k_t, \tilde{\E}[B^k_\mu(\theta_t)(\tilde{X}_t)\tilde{Y}^\epsilon_t] \otimes Y^\epsilon_t+Y^\epsilon_t \otimes \tilde{\E}[B^k_\mu(\theta_t)(\tilde{X}_t)\tilde{Y}^\epsilon_t] \big\rangle    \\
            &-\frac{1}{2}\big\langle \tilde{\E} \left[H_{\mu\mu}(\tilde{\theta}_t)(X_t,\hat{X}_t)\right]
            +2 H_{x\mu}(\hat{\theta}_t)(X_t)
            + P^* _t A_\mu(\theta_t)(\hat{X}_t)
            + P_t A_\mu(\theta_t)(\hat{X}_t)\\
            & + B^*_x(\theta_t) P_t B_\mu(\theta_t)(\hat{X}_t) 
            + \tilde{\E}[B^*_\mu(\tilde{\theta}_t)(X_t)\tilde{P}_t B_\mu(\tilde{\theta}_t)(\hat{X}_t)]
            + B^*_x(\theta_t) P^*_t B_\mu(\theta_t)(\hat{X}_t) \\
            &+ \sum_{k=1}^d \Big(
            Q^{k,*}_t B^k_\mu(\theta_t)(\hat{X}_t)
            +Q^{k}_t B^k_\mu(\theta_t)(\hat{X}_t), Y^\epsilon_t \otimes \hat{Y}^\epsilon_t\big\rangle \dt \Bigg]+o(\epsilon),\\
            = & \bar{\E} \left[ \int_0^T \delta H(t)
            +\frac{1}{2}\big\langle P_t, \delta B(t)\delta B(t)^*\big\rangle \dt\right] +o(\epsilon).
        \end{align*}
    \end{proof}
    \begin{remark}\label{remark:Discussion at the end}
        In the above, a second-order Taylor expansion is necessary, as the order of the variational processes is low (cf. Lemma \ref{lemma:Estimates on Y and Z} and \ref{lemma:estimate on DeltaX-Y-Z}). 
        In the standard SDE setting, \citet{Peng1990} wrote the single(!) appearing quadratic term, coming from the second derivative with respect to the state, as a product of tensors and introduced the first-level second-order adjoint process $P$ to dualize it (cf. \eqref{eq:first second duality}). 
        In our case, this same dualization happens from \eqref{eq:Second Duality Replacement} to \eqref{eq:Third Duality Replacement}, but in the McKean-Vlasov setting even more quadratic terms appear from the (mixed) derivatives with respect to the additional distribution variable. 
        By the nature of the Lions derivative, the representation of these terms involve two independent copies of the variational processes, which makes it impossible to dualize these terms with $P$. 
        \citet{Buckdahn2016} avoided this problem by proving sharper estimates for the first variational process under the expectation (Proposition 4.3 therein), implying sharper estimates for the terms containing independent copies. 
        As a result they could stop their derivation with equation \eqref{eq:Third Duality Replacement}. 
        
        In contrast, we make the dualization of all quadratic terms (which in view of Lemma \ref{lemma:Estimates on Y and Z} and \ref{lemma:estimate on DeltaX-Y-Z} are not apriori of lower-order) possible, by introducing the second-level second-order adjoint process, thus avoiding the need to provide sharper estimates.
        
        The first-level second-order adjoint state can be seen as dualizing the homogeneous part of the first variational equation, while the second-level second-order adjoint state dualizes the heterogeneous part. 
        As we will see below, the third adjoint process, does not appear in the maximum principle, it is only used as a means to get there.
    \end{remark}
    Using \eqref{eq:Final Expansion of Cost Functional}, we immediately, by Lebesgue's differentiation theorem, get the following maximum principle, which was also already shown in \cite{Buckdahn2016}.
    \begin{corollary}[Peng's Maximum Principle]
        Under Assumption \ref{Assumption:C^{2,1}}, if $\alpha$ is optimal, then for Lebesgue almost every $t\in[0,T]$ and $\P$-almost surely, for all $u\in U$,
        \begin{align*}
            0\leq & H(t,X_t,\mu_t,u,p_t,q_t) 
            - H(t,X_t,\mu_t,\alpha_t,p_t,q_t) \\
            &+ \frac{1}{2} \langle P_t,(B(t,X_t,\mu_t,u)-B(t,X_t,\mu_t,\alpha_t))(B(t,X_t,\mu_t,u)-B(t,X_t,\mu_t,\alpha_t))^*\rangle.
        \end{align*}
    \end{corollary}

\section*{Acknowledgments}
WS and JBS acknowledges support from DFG CRC/TRR 388 'Rough Analysis, Stochastic Dynamics and Related Fields', Projects A10 and B09.

\setcitestyle{numbers}
\bibliographystyle{unsrtnat}
\bibliography{refs}

\end{document}